\documentclass[12pt,a4paper,leqno]{amsart}
\usepackage[utf8]{inputenc}

\usepackage{amssymb,amsfonts,mathtools,mathrsfs}   
\usepackage{graphicx}
\usepackage{placeins}
\usepackage{enumerate}
\usepackage{tikz}
\usepackage{fullpage}
\usepackage{setspace}
\usepackage{xcolor}
\usepackage{hyperref}

\usepackage[
backend=biber,
style=alphabetic,
maxalphanames=10,
maxnames=10,
]{biblatex}

\usepackage{xcolor}
\usepackage{hyperref}

\usepackage[
  margin=1.5cm,
  includefoot,
  footskip=30pt,
]{geometry} 

\newtheorem{thm}{Theorem}[section]
\newtheorem{cor}[thm]{Corollary} 
\newtheorem{lem}[thm]{Lemma} 
\newtheorem{prop}[thm]{Proposition}

\theoremstyle{definition}

\theoremstyle{remark}
\newtheorem{remark}[thm]{Remark}

\numberwithin{equation}{section}

\usepackage{comment}
\excludecomment{notetoself}

\newcommand{\R}{\mathbb{R}}

\newcommand{\N}{\mathbb{N}}

\newcommand{\Sph}{\mathbb{S}}

\newcommand{\abs}[1]{ \lvert #1 \rvert}

\newcommand{\supp}{\operatorname{supp}}

\newcommand{\PV}{\operatorname{P.V.}}

\newcommand{\dd}{\mathop{}\!d} 

\def\XXint#1#2#3{{\setbox0=\hbox{$#1{#2#3}{\int}$}
\vcenter{\hbox{$#2#3$}}\kern-.5\wd0}}

\newenvironment{PDE}
	{ \left \{
	\begin{array}{r@{ \ }l @{\quad \: \;} l}
	}
	{
	\end{array} \right . 
	}

\begin{document}

\title[The role of antisymmetric functions
in nonlocal equations]{The role of antisymmetric functions
in nonlocal equations}

\author[S. Dipierro]{Serena Dipierro}
\address{Serena Dipierro: Department of Mathematics and Statistics, The University of Western Australia, 35 Stirling Highway, Crawley, Perth, WA 6009, Australia}
\email{serena.dipierro@uwa.edu.au}

\author[G. Poggesi]{Giorgio Poggesi}
\address{Giorgio Poggesi: Department of Mathematics and Statistics, The University of Western Australia, 35 Stirling Highway, Crawley, Perth, WA 6009, Australia}
\email{giorgio.poggesi@uwa.edu.au}

\author[J. Thompson]{Jack Thompson}
\address{Jack Thompson: Department of Mathematics and Statistics, The University of Western Australia, 35 Stirling Highway, Crawley, Perth, WA 6009, Australia}
\email{jack.thompson@research.uwa.edu.au}

\author[E. Valdinoci]{Enrico Valdinoci}
\address{Enrico Valdinoci: Department of Mathematics and Statistics, The University of Western Australia, 35 Stirling Highway, Crawley, Perth, WA 6009, Australia}
\email{enrico.valdinoci@uwa.edu.au}


\date{\today}

\dedicatory{}

\begin{abstract}
We use a Hopf-type lemma for antisymmetric super-solutions to the Dirichlet problem for the fractional Laplacian with zero-th order terms,
in combination with the method of moving planes, to prove symmetry for the {\it semilinear fractional parallel surface problem}. That is, we prove that non-negative solutions to semilinear Dirichlet problems for the fractional Laplacian in a bounded open set $\Omega \subset \R^n$ must be radially symmetric if one of their level surfaces is parallel to the boundary of $\Omega$; in turn, $\Omega$ must be a ball.

Furthermore, we discuss maximum principles and the Harnack inequality for antisymmetric functions in the fractional setting and provide counter-examples to these theorems when only `local' assumptions are imposed on the solutions. 
The construction of these counter-examples relies on an approximation result that states that
`all antisymmetric functions are locally antisymmetric and \(s\)-harmonic up to a small error'.
\end{abstract}

\maketitle

\section{Introduction and results}
In this paper, we are concerned with maximum principles for antisymmetric solutions to linear, nonlocal partial differential equations (\textsc{PDE}) where the nonlocal term is given by the fractional Laplacian and with the application of such maximum principles to the method of moving planes.

Since its introduction by Alexandrov \cite{MR0150709,MR143162} and the
subsequent refinement by Serrin~\cite{MR333220}, the method of moving planes has become a standard technique for the analysis of overdetermined \textsc{PDE}, and for proving symmetry and monotonicity results for solutions, see \cite{MR544879,MR634248,MR2293958,MR2366129,MR3040677,MR3189604,MR3385189,MR3454619}. In recent years, nonlocal \textsc{PDE}, particularly the ones involving the fractional Laplacian, have received a lot of attention, due to their consolidated ability to model complex physical phenomena, as well as their rich analytic theory. The method of moving planes has been also adapted in several works so that it may be applied to nonlocal problems, see
\cite{MR3395749,MR3827344,MR3881478,MR3836150,MR4313576,ciraolo2021symmetry}.

The method of moving planes typically leads to 
considering the difference between a given solution and its reflection, and such a difference is an antisymmetric function. However, in the local case, `standard' maximum principles
can be directly applied to the moving plane method,
without specifically relying on the notion of
antisymmetric functions, since the ingredients required
are all of `local' nature and can disregard the `global' geometric
properties of the objects considered. Instead,
in the nonlocal world, new maximum principles need to be developed which take into account the extra symmetry, to
compensate for global control of the function under consideration,
for example see \cite{MR3395749}. 

We will now state our main results.
Let \(s\in (0,1)\) and~\(n\geqslant 1\) be an integer. The fractional Laplacian \((-\Delta)^s\) is defined by \begin{align*}
(-\Delta)^s u(x) &= c_{n,s}\PV \int_{\R^n} \frac{u(x)-u(y)}{\vert x -y \vert^{n+2s}} \dd y
\end{align*} where \begin{align}
c_{n,s} = \frac{s 4^s \Gamma \big (\frac{n+2s}{2} \big )}{\Gamma(1-s) } \label{tK7sb}
\end{align}  is a positive normalisation constant and \(\PV\) refers to the Cauchy principle value. Moreover, let~\(\R^n_+ = \{ x\in \R^n \text{ s.t. } x_1>0\}\) and \(B_1^+ = B_1 \cap \R^n_+\). 

Our first result proves that `all antisymmetric functions are locally antisymmetric and \(s\)-harmonic up to a small error'. This is the analogue of \cite[Theorem 1]{MR3626547} for antisymmetric functions. 

\begin{thm} \label{u7dPW} Suppose that \(k \in \N\) and that \(f \in C^k(\overline{B_1})\) is \(x_1\)-antisymmetric. For all \(\varepsilon>0\), there exist a smooth \(x_1\)-antisymmetric function \(u:\R^n \to \R\) and a real number \(R>1\) such that \(u\) is \(s\)-harmonic in \(B_1\), \(u=0\) in \(\R^n\setminus B_R\), and \begin{align*}
\| u - f \|_{C^k(B_1)}<\varepsilon .
\end{align*}
\end{thm}

When we say that \(u\) is \(x_1\)-antisymmetric we mean that \( u(-x_1,\dots,x_n)=-u(x_1,\dots,x_n) \) for all~\(x\in \R^n\). In \S\ref{9Hmeh} we will give a more general definition of antisymmetry for an arbitrary plane. As an application of Theorem \ref{u7dPW}, we can construct some counter-examples to the local Harnack inequality and strong maximum principle for antisymmetric \(s\)-harmonic functions, as follows:

\begin{cor}\label{X96XH}
There exist \(\Omega \subset \R_+^n\) and \(\Omega'\subset \subset \Omega\) such that, for all \(C>0\), there exists an \(x_1\)-antisymmetric function \(u:\R^n \to \R\) that is non-negative and \(s\)-harmonic in \(\Omega\), and \begin{align*}
\sup_{\Omega'} u > C \inf_{\Omega'} u .
\end{align*}
\end{cor}

\begin{cor} \label{KTBle} There exist \(\Omega \subset \R^n_+\) and an \(x_1\)-antisymmetric function \(u:\R^n \to \R\) that is non-negative and \(s\)-harmonic in \(\Omega\), is equal to zero at a point in \(\Omega\), but is not identically zero in \(\Omega\).
\end{cor}

For our second result, we use
a Hopf-type lemma (see the forthcoming Proposition~\ref{lem:FLIFu}) in combination with the method of moving planes to establish symmetry for the {\it semilinear fractional parallel surface problem}, which is described in what follows. 
%
%
Suppose that \(G\subset \R^n \) is open and bounded, and let~\(B_R\) denote the ball of radius \(R>0\) centered at 0. Let \(\Omega = G+B_R\), where, given sets \(A,B \subset \R^n\), \(A+B\) denotes the `Minkowski sum' of \(A\) and \(B\), defined as \begin{align*}
A+B = \{a+b \text{ s.t. } a\in A, b\in B\}. 
\end{align*} Consider the semilinear fractional equation \begin{align}
\begin{PDE}
(-\Delta)^s u &= f(u) &\text{in }\Omega, \\
u&=0 &\text{in }\R^n \setminus \Omega,\\
u&\geqslant 0 &\text{in } \Omega ,
\end{PDE} \label{rzZmb}
\end{align} with the overdetermined condition \begin{align}
u = c_0 \qquad \text{on } \partial G, \label{lnARf}
\end{align} for some (given) \(c_0\geqslant 0\). The \emph{semilinear fractional parallel surface problem} asks the following question:
for which~\(G\) does the problem~\eqref{rzZmb}-\eqref{lnARf} admit a non-trivial solution? The answer is that \(G\) must be a ball. More specifically, we have the following
result:

\begin{thm} \label{CccFw} Suppose that \(G\) is a bounded open set in \(\R^n\) with \(C^1\) boundary. Let~\(\Omega=G+B_R\), \(f:\R \to \R\) be locally Lipschitz, and \(c_0 \geqslant 0\). Furthermore, assume that there exists a non-negative function \(u \in C^s(\R^n)\) that is not identically zero and satisfies, in the pointwise sense,\begin{align} 
\begin{PDE}
(-\Delta)^s u &= f(u) &\text{in } \Omega,\\ u&=0 &\text{in } \R^n \setminus \Omega,  \\ u &= c_0 & \text{on } \partial G .
\end{PDE} \label{zcm6a}
\end{align} Then \(u\) is radially symmetric, \(u>0\) in \(\Omega\), and \(\Omega\) (and hence \(G\)) is a ball. 
\end{thm}

In the local setting (i.e., \(s=1\)), symmetry and
%
%
stability results for \eqref{zcm6a} were obtained in \cite{MR3420522} and \cite{MR3481178}.
%
%
Such analysis was motivated by the study of invariant isothermic surfaces of a nonlinear nondegenerate fast diffusion equation (see \cite{MR2629887}). 

More recently in \cite{ciraolo2021symmetry}, the nonlocal case \(s\in (0,1)\) was addressed 
and the analogue to Theorem \ref{CccFw}, as well as its stability generalization, were proved in the particular case \(f \equiv 1\); those results were obtained by making use of the method of moving planes as well as a (quantitative) Hopf-type lemma (see \cite[formula~(3.3)]{ciraolo2021symmetry}), which could be obtained as an application of the boundary Harnack inequality for antisymmetric $s$-harmonic functions proved in \cite[Lemma~2.1]{ciraolo2021symmetry}. The arguments used in~\cite{ciraolo2021symmetry} to establish such a boundary Harnack inequality rely on the explicit knowledge of the fractional Poisson kernel for the ball. However, due to the general nonlinear term~\(f\) in \eqref{zcm6a}, here the method of moving planes leads to considering a linear equation involving the fractional Laplacian but with zero-th order terms for which no Poisson formula is available. To overcome this conceptual difficulty, we provide the following Hopf-type lemma which allows zero-th order terms.

\begin{prop} \label{lem:FLIFu}  Suppose that \(c \in L^\infty(B_1^+)\),  \(u \in  H^s(\R^n) \cap C(B_1)\) is \(x_1\)-antisymmetric, and satisfies \begin{align}
\begin{PDE}
(-\Delta)^su +cu &\geqslant 0 &\text{in } B_1^+, \\
u&\geqslant 0 &\text{in } \R^n_+ ,\\
u&>0 &\text{in } B_1^+. 
\end{PDE} \label{YEL36}
\end{align} Then \begin{align}
\liminf_{h\to 0} \frac{u(he_1)} h > 0 . \label{FLBHT}
\end{align} 
\end{prop}

This result has previously been obtained in \cite[Proposition 2.2]{MR3937999} though the proof uses a different barrier. 

Proposition \ref{lem:FLIFu} establishes that antisymmetric solutions to \eqref{YEL36} must be bounded from below by a positive constant multiple of \(x_1\) close to the origin. At a first glance this is surprising as solutions to \eqref{YEL36} are in general only \(s\)-H\"older continuous up to the boundary, as proven in \cite{MR3168912}. Hence, it would appear more natural to 
consider \(\liminf_{h\to 0} \frac{u(he_1)} {h^s}\) instead of \eqref{FLBHT}, see for example \cite[Proposition 3.3]{MR3395749}. However, the (anti)symmetry of \(u\) means that Proposition~\ref{lem:FLIFu} is better understood as an interior estimate rather than a boundary estimate. 
Indeed, we stress that, differently from the classical Hopf lemma,
Proposition~\ref{lem:FLIFu} is not concerned with the growth of a solution from a boundary point, but mostly with the growth from a reflection point in the antisymmetric setting (notice specifically that~\eqref{FLBHT} provides a linear growth from the origin, which is
an interior point). In this sense, the combination of the antisymmetric geometry
of the solution and the fractional nature of the equation
leads to the two structural differences between Proposition~\ref{lem:FLIFu} and several other Hopf-type lemmata available in the literature for the nonlocal
case, namely the linear (instead of \(s\)-H\"older) growth
and the interior (instead of boundary) type of statement.

A similar result to Proposition~\ref{lem:FLIFu} was obtained in \cite[Theorem 1]{MR3910421} for the entire half-space instead of \(B_1^+\). This entails that~\cite[Theorem 1]{MR3910421} was only applicable (via the method of moving planes) to symmetry problems posed in all of \(\R^n\) whereas our result can be applied to problems in bounded domains. Moreover, \cite[Theorem 1]{MR3910421} is proven by contradiction while the proof of Proposition~\ref{lem:FLIFu}, in a similar vein to the original Hopf lemma, relies on the construction of an appropriate barrier from below, see Lemma \ref{SaBD4}. Hence, using a barrier approach also allows us to obtain a quantitative version of Proposition~\ref{lem:FLIFu} leading to quantitative estimates for the stability of Theorem \ref{CccFw}, which we plan to address in an upcoming paper.

It is also natural ask whether the converse to Theorem~\ref{CccFw} holds. We recall that if $\Omega$ is known to be a ball, i.e., $\Omega:=B_\rho$, then the radial symmetry of a bounded positive solution $u$ to
\begin{equation}\label{eq:Classical Dirichlet}
\begin{PDE}
-\Delta u &= f(u) & \text{in }  B_\rho\\
u&=0 &\text{on } \partial B_\rho
\end{PDE}
\end{equation}
is guaranteed for a large class of nonlinearities $f$. For instance, this is the case for $f$ locally Lipschitz, by the classical result obtained via the method of moving planes by Gidas-Ni-Nirenberg \cite{MR634248,MR544879}. Many extensions of Gidas-Ni-Nirenberg result can be found in the literature, including generalizations to the fractional setting (see, e.g., \cite{MR2114412}).
We also recall that an alternative method pioneered by P.-L. Lions in \cite{MR653200} provides symmetry of nonnegative bounded solutions to \eqref{eq:Classical Dirichlet} for (possibly discontinuous) nonnegative nonlinearities; we refer to \cite{MR3003296,MR1382205,MR2019179,MR4380032} for several generalizations.

\medskip

The paper is organised as follows. In Section \ref{eBxsh} we recall some standard notation, as well as provide some alternate proofs for the weak and strong maximum principle for antisymmetric functions. Moreover, we will prove Theorem \ref{u7dPW} and subsequently prove Corollary \ref{X96XH} and Corollary \ref{KTBle}. In Section \ref{R9GxY}, we will prove Proposition~\ref{lem:FLIFu} and in Section \ref{Svlif} we will prove Theorem \ref{CccFw}. In Appendix \ref{ltalz} we give some technical lemmas needed in the paper. 

\section{Maximum principles and counter-examples} \label{eBxsh}

\subsection{Definitions and notation} \label{9Hmeh}
Let \(n \geqslant 1\) be an integer, \(s\in (0,1)\), and \(\Omega\subset \R^n\) be a bounded open set.  The fractional Sobolev space \(H^s(\R^n)\) is defined as \begin{align*}
H^s(\R^n) = \bigg \{ u \in L^2(\R^n) \text{ s.t. } [ u ]_{H^s(\R^n)} < \infty \bigg \}
\end{align*} where \([ u ]_{H^s(\R^n)}\) is the Gagliardo semi-norm
\begin{align*}
[ u ]_{H^s(\R^n)} &= \bigg ( \int_{\R^n} \int_{\R^n} \frac{\vert u(x)-u(y)\vert^2}{\vert x -y \vert^{n+2s}} \dd x \dd y \bigg )^{\frac12}. 
\end{align*} As usual we identify functions that are equal except on a set of measure zero. It will also be convenient to introduce the space \begin{align*}
\mathcal H^s_0(\Omega) = \{ u \in H^s (\R^n) \text{ s.t. } u = 0 \text{ in } \R^n \setminus \Omega \}. 
\end{align*} Suppose that \(c:\Omega \to \R\) is a measurable function and that \(c\in L^\infty(\Omega)\).  A function \(u \in H^s(\R^n)\) is a \emph{weak} or \emph{generalised} solution of \((-\Delta)^s u +cu =0\) (resp. \(\geqslant 0, \leqslant 0\)) in \(\Omega\) if \begin{align*}
\mathfrak L (u,v) := \frac{c_{n,s}}2 \iint_{\R^n \times \R^n} \frac{(u(x)-u(y))(v(x)-v(y))}{\abs{x-y}^{n+2s}} \dd x \dd y  + \int_{\Omega} c(x) u(x) v(x) \dd x =0 \, (\geqslant 0, \leqslant 0)
\end{align*} for all \(v \in \mathcal H^s_0(\Omega) \), \(v \geqslant 0\). Also, it will be convenient to use the notation \begin{align*}
\mathcal E (u,v) = \frac{c_{n,s}}2 \iint_{\R^n \times \R^n} \frac{(u(x)-u(y))(v(x)-v(y))}{\abs{x-y}^{n+2s}} \dd x \dd y
\end{align*} for each \(u,v \in H^s(\R^n)\). 

A function \(u :\R^n \to \R\) is \emph{antisymmetric} with respect to a plane \(T\) if \begin{align*}
u(Q_T(x)) = -u(x) \qquad \text{for all } x\in \R^n
\end{align*} where \(Q_T:\R^n \to \R^n\) is the reflection of \(x\) across \(T\). A function \(u\) is \emph{\(x_1\)-antisymmetric} if it is antisymmetric with respect to the plane \(T=\{x_1=0\}\). In the case \(T=\{x_1=0\}\), \(Q_T\) is given explicitly by \(Q_T(x) = x-2x_1e_1\). When it is clear from context what \(T\) is, we will also write~\(Q(x)\) or~\(x_\ast\) to mean~\(Q_T(x)\). 

We will also make use of standard notation: the upper and lower half-planes are given by \(\R^n_+= \{ x \in \R^n \text{ s.t } x_1>0\}\) and \(\R^n_-= \{ x \in \R^n \text{ s.t } x_1<0\}\) respectively, and, for all \(r>0\), the half-ball of radius \(r\) is given by \(B_r^+ = \R^n_+ \cap B_r\). We also denote the positive and negative part of a function~\(u\) by~\(u^+(x) = \max \{u(x),0\}\)
and~\(u^-(x) = \max \{-u(x),0\}\). Furthermore, the characteristic function of a set \(A\) is given by \begin{align*}
\chi_A(x) &= \begin{cases}
1, &\text{if } x\in A, \\
0,&\text{if } x\not\in A. 
\end{cases}
\end{align*}

\subsection{Maximum principles}

We now present some results of maximum principle type for antisymmetric functions.

\begin{prop}[A weak maximum principle] \label{aKht4} Suppose that \(\Omega\) is a strict open subset of \(\R^n_+\), and~\(c :\Omega \to \R\) ia a measurable and non-negative function. Let \(u \in H^s(\R^n)\) be \(x_1\)-antisymmetric. If \(u\) satisfies \((-\Delta)^su+cu\leqslant0\) in \(\Omega\) then \begin{align*}
\sup_{\R^n_+} u \leqslant \sup_{\R^n_+\setminus \Omega} u^+ . 
\end{align*} Likewise, if \(u\) satisfies \((-\Delta)^su+cu\geqslant0\) in \(\Omega\) then \begin{align*}
\inf_{\R^n_+} u \geqslant - \sup_{\R^n_+ \setminus \Omega } u^- . 
\end{align*}
\end{prop}

In our proof, we use that $c \ge 0$. We refer, e.g., to \cite[Proposition 3.1]{MR3395749} for a weak maximum principle where the non-negativity of $c$ is not required.

\begin{proof}[Proof of Proposition \ref{aKht4}]
It is enough to prove the case \((-\Delta)^su+cu\leqslant0\) in \(\Omega\). Suppose that \(\ell := \sup_{\R^n_+ \setminus \Omega} u^+ <+\infty\), otherwise we are done.

For each \(v\in H^s_0(\Omega)\), \(v \geqslant 0\),  rearranging  \(\mathfrak L (u,v) \leqslant 0\) gives \begin{align}
\mathcal E(u,v) \leqslant - \int_\Omega c(x) u(x) v(x) \dd x \leqslant 0  \label{FQBj4}
\end{align} provided that \(uv \geqslant 0\) in \(\Omega\). 

Let \(w(x):=u(x)-\ell\) and \(v (x) := (u(x)-\ell)^+\chi_{\R^n_+}(x)=w^+(x)\chi_{\R^n_+}(x)\). Note that for all \(x\in \Omega\),\begin{align*}
u(x)v(x) &=  ((u(x)-\ell)^+)^2 + \ell (u(x)-\ell)^+ \geqslant 0. 
\end{align*} Here we used that \((u-\ell)^+(u-\ell)^-=0\). To prove the proposition, it is enough to show that \(v=0\) in \(\R^n\). As \(u(x)-u(y)=w(x)-w(y)\), expanding then using that \(w^+ (x) w^-(x)=0\) gives  \begin{align*}
(u(x)-u(y))(v(x)-v(y)) &= ((v(x)-v(y))^2 +v(x)(v(y)-w(y)) +v(y) (v(x)-w(x))
\end{align*} for all \(x,y \in \R^n\). Hence, \begin{align*}
\mathcal E (u,v) &= \mathcal E (v,v) + c_{n,s} \iint_{\R^n\times \R^n} \frac{v(x)(v(y)-w(y))}{\vert x- y \vert^{n+2s} } \dd x \dd y.
\end{align*} Observe that\begin{align*}
\iint_{\R^n\times \R^n} \frac{v(x)(v(y)-w(y))}{\vert x- y \vert^{n+2s} } \dd x \dd y &= \iint_{\R^n_+\times \R^n_+} \frac{w^+(x)w^-(y)}{\vert x- y \vert^{n+2s} } \dd x \dd y-\iint_{\R^n_-\times \R^n_+} \frac{w^+(x)w(y)}{\vert x- y \vert^{n+2s} } \dd x \dd y.
\end{align*} Making the change of variables \(y \to y_\ast\), where \(y_\ast\) denotes the reflection of \(y\) across the plane \(\{x_1=0\}\), then writing \(w=w^+-w^-\) gives \begin{align*}
\iint_{\R^n\times \R^n} \frac{v(x)(v(y)-w(y))}{\vert x- y \vert^{n+2s} } \dd x \dd y &=\iint_{\R^n_+\times \R^n_+} \frac{w^+(x)w^-(y)}{\vert x- y \vert^{n+2s} } \dd x \dd y+\iint_{\R^n_+\times \R^n_+} \frac{w^+(x)(w(y)+2\ell)}{\vert x- y_\ast \vert^{n+2s} } \dd x \dd y \\
&= \iint_{\R^n_+\times \R^n_+} w^+(x)w^-(y) \bigg ( \frac1{\vert x- y \vert^{n+2s}}-\frac1{\vert x- y_\ast \vert^{n+2s}} \bigg )\dd x \dd y \\
& \qquad + \iint_{\R^n_+\times \R^n_+} \frac{w^+(x)(w^+(y)+2\ell)}{\vert x- y_\ast \vert^{n+2s} } \dd x \dd y \\
&\geqslant 0. 
\end{align*} Thus, \begin{align}
\mathcal E(u,v) \geqslant \mathcal E  (v,v) \geqslant 0 . \label{vndMm}
\end{align} Combining \eqref{FQBj4} and \eqref{vndMm} gives that \(\mathcal E(v,v) =0\) which implies that \(v\) is a constant in \(\R^n\). Since \(v=0\) in \(\R^n_- \), we must have that \(v=0\) as required. 
\end{proof}

\begin{remark}
It follows directly from Proposition \ref{aKht4} that if \(u\) satisfies \((-\Delta)^s u +cu =0\) in \(\Omega\) then \begin{align*}
\sup_{\R^n_+} \vert u \vert \leqslant \sup_{\R^n_+\setminus \Omega} \vert  u \vert . 
\end{align*}
\end{remark}

Next, we prove the strong maximum principle for antisymmetric functions. This result is not new in the literature, see \cite[Corollary 3.4]{MR3395749}; however, Proposition \ref{fyoaW} provides an alternate elementary proof. 

\begin{prop}[A strong maximum principle] \label{fyoaW} Let \(\Omega\subset \R^n_+\) and \(c :\Omega \to \R\). Suppose that \(u :\R^n \to \R\) is \(x_1\)-antisymmetric and \((-\Delta)^su\) is defined pointwise in \(\Omega\). If \(u\) satisfies \((-\Delta)^su+cu \geqslant 0\) in \(\Omega\), and \(u\geqslant 0\) in \(\R^n_+\) then either \(u>0\) in \(\Omega\) or \(u\) is zero almost everywhere in \(\R^n\). 
\end{prop}

\begin{remark}Note that in Proposition~\ref{fyoaW}, as in \cite[Corollary 3.4]{MR3395749}, no sign assumption is required on \(c\). The reason Proposition~\ref{fyoaW} holds without this assumption is due to the sign assumption on \(u\) which is not usually present in the statement of the strong maximum principle.
\end{remark}

\begin{proof}[Proof of Proposition~\ref{fyoaW}]
Due to the antisymmetry of \(u\), by a change of variables we may write \begin{align}
\begin{aligned}
(-\Delta)^s u(x) &= c_{n,s} \PV \int_{\R^n_+} \bigg ( \frac 1 {\vert x - y \vert^{n+2s}} - \frac 1 {\vert x_\ast - y \vert^{n+2s}} \bigg ) (u(x)-u(y)) \dd y  \\
&\qquad \qquad + 2c_{n,s}\int_{\R^n_+} \frac{u(x)}{\vert x_\ast - y \vert^{n+2s}} \dd y
\end{aligned} \label{u9I7E}
\end{align} where \(x_\ast=x-2x_1e_1\) is the reflection of \(x\) across \(\{x_1=0\}\). Suppose that there exists \(x_\star\in \Omega\) such that \(u(x_\star)=0\). On one hand,  \eqref{u9I7E} implies that \begin{align}
(-\Delta)^s u(x_\star) &= -c_{n,s} \PV \int_{\R^n_+} \bigg ( \frac 1 {\vert x_\star - y \vert^{n+2s}} - \frac 1 {\vert (x_\star)_\ast - y \vert^{n+2s}} \bigg )u(y) \dd y \label{XXP2e}
\end{align} where we used that \(u\) is antisymmetric. Since \begin{align*}
\frac 1 {\vert x_\ast - y \vert^{n+2s}} < \frac 1 {\vert x - y \vert^{n+2s}}, \qquad \text{for all } x,y\in \R^n_+, \, x\neq y,
\end{align*} equation \eqref{XXP2e} implies that \((-\Delta)^s u(x_\star) \leqslant 0\) with equality if and only if \(u\) is identically zero almost everywhere in \(\R^n_+\). On the other hand, \((-\Delta)^su(x_\star) = (-\Delta)^s u(x_\star) +c(x_\star) u(x_\star) \geqslant 0 \). Hence, \((-\Delta)^su(x_\star)=0\) so \(u\) is identically zero a.e. in \(\R^n_+\).
\end{proof}

\subsection{Counter-examples}
The purpose of this subsection is to provide counter-examples to the classical Harnack inequality and the strong maximum principle for antisymmetric \(s\)-harmonic functions. A useful tool to construct such functions is the following antisymmetric analogue of Theorem 1 in \cite{MR3626547}. This proves that `all antisymmetric functions are locally antisymmetric and \(s\)-harmonic up to a small error'. 

\begin{proof}[Proof of Theorem \ref{u7dPW}]
Due to \cite[Theorem 1.1]{MR3626547} there exist \(R>1\) and \(v\in H^s (\R^n) \cap C^s(\R^n)\) such that \((-\Delta)^sv=0\) in \(B_1\), \(v=0\) in \(\R^n \setminus B_R\), and \( \| v - f\|_{C^k(B_1)} < \varepsilon \). In fact, by a mollification argument, we may take \(v\) to be smooth, see Remark \ref{BQkXm} here below.

Let \(x_\ast\) denote the reflection of \(x\) across \( \{ x_1=0 \}\), and \begin{align*}
u(x) = \frac12 \big ( v(x) - v(x_\ast) \big ) . 
\end{align*} It is easy to verify that \((-\Delta)^s u = 0\) in \(B_1\) (and of course \(u=0\) in \(\R^n \setminus B_R\)). By writing \begin{align*}
u(x) -f(x) &= \frac12 \big (v(x) - f(x) \big ) + \frac12 \big ( f(x_\ast) - v(x_\ast) \big ) \qquad\text{for all } x \in B_1
\end{align*} we also obtain \begin{equation*}
\| u - f\|_{C^k(B_1)} \leqslant \| v - f\|_{C^k(B_1)} < \varepsilon . \qedhere
\end{equation*}
\end{proof}

\begin{remark} \label{BQkXm}
Fix \(\varepsilon>0\). If \(f \in C^k (\overline{B_1} )\) there exists \(\mu >0\) such that \(f\in C^k(B_{1+2\mu})\). Rescaling then applying \cite[Theorem 1.1]{MR3626547}, there exist \(\tilde R>1+\mu\) and \(\tilde v\in H^s (\R^n) \cap C^s(\R^n)\) such that \((-\Delta)^s \tilde v=0\) in \(B_{1+\mu}\), \(\tilde v=0\) in \(\R^n \setminus B_{\tilde R }\), and \( \| \tilde v - f\|_{C^k(B_{1+\mu})} < \varepsilon/2 \).

Let \(\eta \in C^\infty_0(\R^n)\) be the standard mollifier \begin{align*}
\eta (x) &= \begin{cases}
C e^{- \frac 1 {\vert x \vert^2-1}} , &\text{if } x \in B_1 ,\\
0, &\text{if } x \in \R^n \setminus B_1,
\end{cases}
\end{align*} with \(C>0\) chosen so that \(\int_{\R^n} \eta(x) \dd x =1\). For each \(\delta>0\), let \(\eta_\delta(x) = \delta^{-n} \eta (x/\delta))\) and  \begin{align*}
\tilde{v}^{(\delta)}(x) := (\tilde v \ast \eta_\delta)(x) = \int_{\R^n} \eta_\delta (x-y) \tilde v (y) \dd y. 
\end{align*} By the properties of mollifiers, \begin{align*}
\| \tilde v^{(\delta)} - f \|_{C^k(B_1)} \leqslant \| \tilde v^{(\delta)} - \tilde v \|_{C^k(B_1)} + \| \tilde v - f \|_{C^k(B_1)} < \varepsilon
\end{align*} provided \(\delta\) is sufficiently small. Since \(\tilde v \in H^s(\R^n) \subset L^2(\R^n)\), it follows from \cite[Proposition 4.18]{MR2759829} that \(\supp \tilde{v}^{(\delta)}(x) \subset  \overline{B_{\tilde{R}+\delta}}\). Moreover, via the Fourier transform, \((-\Delta)^s \tilde{v}^{(\delta)} = \big ( (-\Delta)^s \tilde{v} \big ) \ast \eta_\delta\), so \cite[Proposition 4.18]{MR2759829} again implies that \( \supp \big ( (-\Delta)^s \tilde{v}^{(\delta)}\big ) \subset \R^n \setminus B_{1+\mu-\delta}\).

Hence, setting \(R=\tilde R+\delta\), \(v=\tilde v^{(\delta)}\), we have constructed a function \(v  \in C^\infty_0(\R^n)\) such that \(v\) is \(s\)-harmonic in \(B_1\), \(v=0\) outside \(B_R\) and \(\| v - f \|_{C^k(B_1)}<\varepsilon\). Moreover, 
$$ v(x_\ast) = \int_{\R^n} \eta_\delta (y) \tilde v (x_\ast- y) \dd y =-\int_{\R^n} \eta_\delta (y) \tilde v (x- y_\ast) \dd y =-v(x)$$ via the change of variables \(z=y_\ast\) and using that \(\eta_\delta(y_\ast) = \eta_\delta (y)\). 
\end{remark}

It is well known that the classical Harnack inequality fails for \(s\)-harmonic functions, see \cite{Kassmann2007clas} for a counter-example. The counter-example provided in \cite{Kassmann2007clas} is not antisymmetric; however, Corollary \ref{X96XH} proves that even with this extra symmetry the classical Harnack inequality does not hold for \(s\)-harmonic functions. Moreover, the construction of the counter-example in Corollary \ref{X96XH} is entirely different to the one in \cite{Kassmann2007clas}. 
 
 \begin{proof}[Proof of Corollary \ref{X96XH}] We will begin by proving the case \(n=1\). Suppose that \(\varepsilon\in (0,1)\), \(I'=(1,2)\),  and \(I=(1/2,5/2)\). Let \(\{ f^{(\varepsilon)}\} \subset C^\infty (\R)\) be a family of odd functions that depend smoothly on the parameter \(\varepsilon\). Moreover, suppose that \begin{align}
 \sup_{I'} f^{(\varepsilon)} = 4 \label{UJd6j}
\end{align}  and \begin{align}
\inf_{I} f^{(\varepsilon)} =\inf_{I'} f^{(\varepsilon)} =2\varepsilon . \label{BR8WJ}
\end{align}  For example, such a family of functions is \begin{align*}
 f^{(\varepsilon)}(x) =  ax +bx^3+cx^5 +dx^7 
 \end{align*} where \begin{align*}
 a = \frac 5{54} (64+5\varepsilon), \quad b= -\frac 1{72}(128+73\varepsilon), \quad c=\frac 1 {36} (-8+23\varepsilon), \quad d=\frac 1 {216}(16-19\varepsilon).
 \end{align*} The functions \(f^{(\varepsilon)}\) are plotted for several values of \(\varepsilon\) in Figure \ref{2FAHv}. 
 
 After a rescaling, it follows from Theorem \ref{u7dPW} that there exists a family \(v^{(\varepsilon)}\) of odd functions that are \(s\)-harmonic in \((-5/2,5/2)\supset I \) and satisfy \begin{align}
 \| v^{(\varepsilon)} - f^{(\varepsilon)} \|_{C((-5/2,5/2))} < \varepsilon . \label{pA7Na}
 \end{align} It follows from \eqref{BR8WJ} and \eqref{pA7Na} that \begin{align*}
\inf_{I} v^{\varepsilon}  \geqslant \varepsilon>0 . 
\end{align*}  Moreover, from \eqref{UJd6j}-\eqref{BR8WJ} and \eqref{pA7Na} , we have that \begin{align*}
\sup_{I'} v^{(\varepsilon)}\geqslant 4-\varepsilon \qquad \text{and} \qquad \inf_{ I'} v^{(\varepsilon)}\leqslant   3\varepsilon . 
\end{align*} Hence, \begin{align*}
\frac{\sup_{I'} v^{(\varepsilon)}}{\inf_{I'} v^{(\varepsilon)}} \geqslant \frac{4-\varepsilon}{3\varepsilon} \to + \infty
\end{align*} as \(\varepsilon \to 0^+\). Setting \(\Omega=I\) and \(\Omega’=I’\) proves the statement in the case \(n=1\).

To obtain the case \(n >1\), set \(\Omega'=I'\times (-1,1)^{n-1}\), \(\Omega=I \times (-2,2)^{n-1}\), and \(u^{(\varepsilon)} (x) = v^{(\varepsilon)} (x_1) \). Using that \begin{align*}
(-\Delta)^s u^{(\varepsilon)} (x) = C (-\Delta)^s_{\R} v^{(\varepsilon)} (x_1) \qquad \text{for all } x\in \R^n 
\end{align*}where \((-\Delta)^s_{\R} \) denotes the fractional Laplacian in one dimension and \(C\) is some constant, see \cite[Lemma 2.1]{MR3536990}, all of the properties of \(v^{(\varepsilon)} \) carry directly over to \(u^{(\varepsilon)} \). 
 \end{proof}

\begin{figure}[h]
\centering
\includegraphics[scale=1]{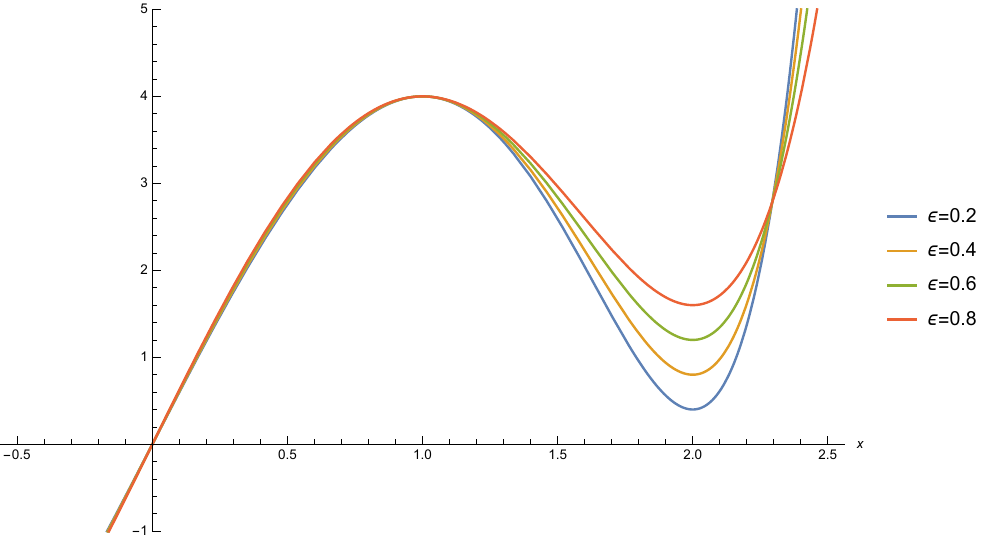} 
\caption{Plot of \(f^{(\varepsilon)}(x) \) for \(\varepsilon=0.2,0.4,0.6,0.8\).}
\label{2FAHv}
\end{figure}

Observe that the family of functions in the proof of Corollary \ref{X96XH} does not violate the classical strong maximum principle. In our next result, Corollary \ref{KTBle}, we use Theorem \ref{u7dPW} to provide a counter-example to the classical strong maximum principle for antisymmetric \(s\)-harmonic functions. This is slightly more delicate to construct than the counter-example in Corollary \ref{X96XH}. Indeed, Theorem \ref{u7dPW} only gives us an \(s\)-harmonic function \(u\) that is \(\varepsilon\)-close to a given function \(f\), so how can we guarantee that \(u=0\) at some point but \(u\) is not negative? 

The idea is to begin with an antisymmetric function \(f\) that has a non-zero minimum away from zero, see Figure \ref{73Iyk}, use Theorem \ref{u7dPW} to get an antisymmetric \(s\)-harmonic function with analogous properties then minus off a known antisymmetric \(s\)-harmonic function until the function touches zero. The proof of Corollary \ref{KTBle} relies on two technical results, Lemma \ref{IyXUO} and Lemma \ref{oujju}, both of which are included in Appendix \ref{ltalz}.

\begin{proof}[Proof of Corollary \ref{KTBle}] By the same argument at the end of the proof of Corollary~\ref{X96XH}, it is enough to take \(n=1\). Let \(\Omega=(0,3)\) and \(f :\R \to \R\) be a smooth, odd function such that \begin{align}
f(x) &\geqslant 1, \text{ for all } x\in [1,3];  \label{cv7Fh}\\
f(x) &\geqslant 3x,\text{ for all } x\in [0,1];  \label{iHz0V}\\
f(2) & =1; \text{ and }  \label{QEb8m} \\
f(3)&= 5. \label{ugncP}
\end{align} For example, such a function is  \begin{align*}
f(x) = -\frac{371 }{43200}x^9+\frac{167 }{1440}x^7-\frac{2681 }{14400}x^5-\frac{4193 }{2160}x^3+\frac{301 }{50} x. 
\end{align*} See Figure \ref{73Iyk} for a plot of \(f\).

By Theorem \ref{u7dPW} with \(\varepsilon =1\), there exists an odd function \(v \in C^\infty_0(\R)\) such that \(v\) is \(s\)-harmonic in~\((-4,4)\) and \begin{align}
\| v-f\|_{C^1((-4,4))} < 1 . \label{wjbkE}
\end{align} Let \(c_0,\zeta_R\) be as in Lemma \ref{oujju} and set \(\zeta := \frac 1 {c_0} \zeta_R. \) By choosing \(R>3\) sufficiently large, we have that  \begin{align}
\frac 3 4 x \leqslant \zeta (x) \leqslant \frac 5 4 x \qquad \text{in }(0,4) . \label{rU1A4}
\end{align} Define \begin{align*}
\phi_t(x) = v(x) - t \zeta (x), \qquad {\mbox{for all~$x\in(-3,3)$ and~$t >0$.}}
\end{align*} We have that, for all \(t>0\), \begin{align*}
(-\Delta)^s \phi_t(x) = 0  \qquad \text{in } (-3,3)
\end{align*} where \((-\Delta)^s=(-\Delta)^s_x \) is the fractional Laplacian with respect to \(x\), and \begin{align*}
x \mapsto \phi_t(x) \text{ is odd.} 
\end{align*} As in Lemma \ref{IyXUO}, let \begin{align*}
m(t) = \min_{x \in [1,3]} \phi_t(x) . 
\end{align*}  From \eqref{cv7Fh} and \eqref{wjbkE} we have that \( \phi_0(x) = v(x) > 0\) in \([1,3]\). Hence, \begin{align*}
m(0) >0. 
\end{align*} Moreover, by \eqref{QEb8m} and \eqref{wjbkE}, \(v(2)<2\). It follows from \eqref{rU1A4} that \begin{align*}
m(4/3) \leqslant \phi_{4/3}(2)=v(2) - \frac 4 3 \zeta (2) <0. 
\end{align*} Since \(m\) is continuous, as shown in Lemma \ref{IyXUO}, the intermediate value theorem implies the existence of \(t_\star \in (0,4/3)\) such that \begin{align*}
m(t_\star) = 0. 
\end{align*} Let \(u := \phi_{t_\star}\). By construction \(u \geqslant 0\) in \([1,3]\). Moreover, since \(u\) is continuous, there exists some \(x_\star \in [1,3]\) such that \(u(x_\star)=0\). In fact, we have that \(x_\star \neq 3\). Indeed, \eqref{ugncP} implies that \(v(3)>4\) and, since \(t_\star < 4/3\) and \(\zeta (3) <9/4\), we obtain \(u(3) >0\). 

All that is left to be shown is that \(u \geqslant 0\) in \((0,1)\). By \eqref{wjbkE}, we have that \begin{align*}
v'(x) > f'(x) - 1, \qquad {\mbox{for all }} x \in (-4,4). 
\end{align*} Since \(v(0)=f(0)=0\), it follows from the fundamental theorem of calculus that \begin{equation}\begin{split}
u(x) &= \int_0^x v'(\tau ) \dd \tau -t_\star \zeta (x) \\
&\geqslant  \int_0^x \big ( f'(\tau)-1 \big ) \dd \tau  - t_\star \zeta (x) \\
&=f(x)-x - t_\star \zeta (x) \label{A0ZNc}
\end{split}\end{equation} provided that~\(x\in (0,3)\). By \eqref{A0ZNc}, \eqref{iHz0V} and \eqref{rU1A4}, \begin{align*}
u(x) \geqslant  \frac 1 3 x \geqslant 0 \qquad \text{in } [0,1] ,
\end{align*}
as desired.
\end{proof}

\begin{figure}[h]
\centering
\includegraphics[scale=1]{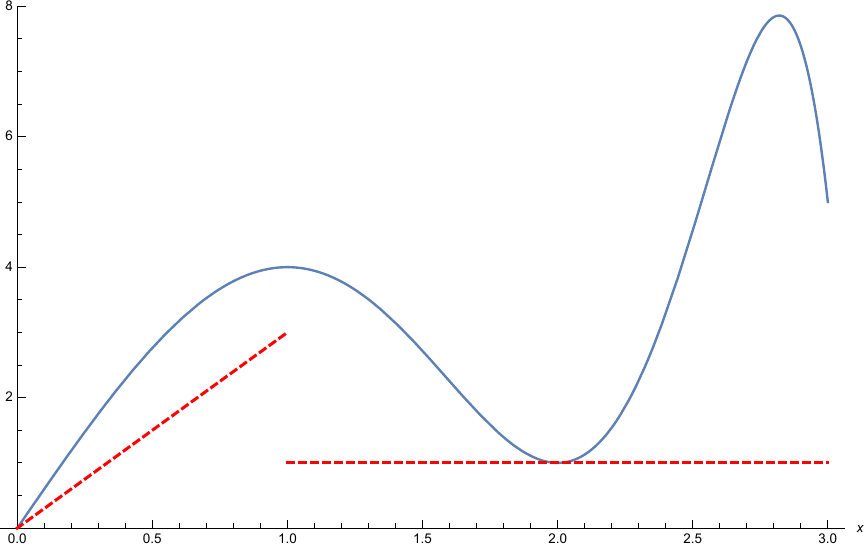} 
\caption{Plot of \(f \) as in the proof of Corollary \ref{KTBle}.}\label{73Iyk}
\end{figure}

\section{A Hopf lemma} \label{R9GxY}
In this section, we prove Proposition~\ref{lem:FLIFu}. The main step in the proof is the construction of the following barrier. 

\begin{lem} \label{SaBD4}
Suppose that \(c \in L^\infty (B_2^+)\). Then there exists an \(x_1\)-antisymmetric function \(\varphi \in C^\infty(\R^n)\) such that \begin{align*}
\begin{PDE}
(-\Delta)^s \varphi +c \varphi &\leqslant 0 &\text{in }B_2^+\setminus B_{1/2}(e_1), \\
\varphi &=0 &\text{in } \R^n_+ \setminus B_2^+ ,\\
\varphi &\leqslant 1 &\text{in } B_{1/2}(e_1),\\
\partial_1 \varphi (0) &>0 . 
\end{PDE}
\end{align*} 
\end{lem}

\begin{proof}
Let \(\zeta\) be a smooth \(x_1\)-antisymmetric cut-off function such that the support of \(\zeta\) is contained in \(B_2\),  \(\zeta \geqslant 0\) in \(B_2^+\), and \(\partial_1\zeta (0)>0\). For example, such a function is \( x_1\eta(x)\) where \(\eta\) is the standard mollifier defined in Remark \ref{BQkXm}. Since \(\zeta\) is smooth with compact support, we have that~\((-\Delta)^s \zeta \in C^\infty (\R^n)\). 

Moreover, \((-\Delta)^s \zeta \) is \(x_1\)-antisymmetric since \(\zeta\) is \(x_1\)-antisymmetric, so it follows that there exists~\(C>0\) such that \begin{align}
(-\Delta)^s \zeta (x) +c \zeta (x) \leqslant C(1+\| c \|_{L^\infty(B_2^+)}) x_1 \qquad \text{in } B_2^+  \label{I6ftM}
\end{align} using that \(c\in L^\infty(B_2^+)\).

Next, let \(\tilde \zeta \) be a smooth \(x_1\)-antisymmetric function such that \(\tilde \zeta  \equiv 1\) in \(B_{1/4}(e_1)\), \(\tilde \zeta  \equiv 0\) in \(\R^n_+ \setminus B_{3/8}(e_1)\), and \(0\leqslant \tilde \zeta  \leqslant 1\) in \(\R^n_+\). Recall that, given an \(x_1\)-antisymmetric function \(u\), the fractional Laplacian of \(u\) can be written as \begin{align*}
(-\Delta)^su(x) &= c_{n,s} \int_{\R^n_+} \bigg ( \frac 1 {\vert x - y\vert^{n+2s}} -  \frac 1 {\vert x_\ast -y\vert^{n+2s}} \bigg ) (u(x)-u(y) ) \dd y + 2 c_{n,s} u(x) \int_{\R^n_+} \frac{\dd y } {\vert x_\ast - y \vert^{n+2s}}.
\end{align*} Hence, for each \(x\in B_2^+ \setminus B_{1/2}(e_1)\),  \begin{align*}
(-\Delta)^s \tilde \zeta (x)&=-c_{n,s}  \int_{B_{3/8}(e_1)} \bigg ( \frac 1 {\vert x - y \vert^{n+2s}} - \frac 1 {\vert x_\ast - y \vert^{n+2s}} \bigg ) \tilde \zeta (y) \dd y.
\end{align*} By the fundamental theorem of calculus, for all \(x\in B_2^+ \setminus B_{1/2}(e_1)\) and~\(y\in B_{3/8}(e_1)\),\begin{align*}
\frac 1 {\vert x - y \vert^{n+2s}} - \frac 1 {\vert x_\ast - y \vert^{n+2s}} &= \frac{n+2s}{2} \int_{\vert x - y \vert^2}^{\vert x_\ast - y \vert^2} \frac{ \dd \tau } {\tau^{\frac{n+2s+2}{2}}}\\
&\geqslant \frac{Cx_1y_1}{\vert x_\ast - y \vert^{n+2s+2}} \\
&\geqslant Cx_1  
\end{align*} with \(C\) depending only on \(n\) and \(s\). Hence, \( (-\Delta)^s\tilde \zeta (x) \leqslant -Cx_1\) in \( B_2^+ \setminus B_{1/2}(e_1)\). Then the required function is given by \(\varphi(x) := \zeta (x) + \alpha \tilde \zeta (x)\) for all \(x\in \R^n\) with \(\alpha>0\) to be chosen later. Indeed, from~\eqref{I6ftM}, we have that \begin{align*}
(-\Delta)^s \varphi +c \varphi \leqslant C (1-\alpha+ \| c \|_{L^\infty(B_2^+)})x_1 \leqslant 0
\qquad\text{in } B_2^+ \setminus B_{1/2}(e_1)
\end{align*} provided that \(\alpha \) is large enough. 
\end{proof}

{F}rom Lemma \ref{SaBD4} the proof of Proposition~\ref{lem:FLIFu} follows easily:

\begin{proof}[Proof of Proposition~\ref{lem:FLIFu}]
Since \(u >0\) in \(B_1^+\), we have that,
for all \(v\in \mathcal H^s_0(B_1^+)\) with~\(v\geqslant0\), \begin{align*}
0\leqslant \mathcal E(u,v) + \int_{B_1^+} c(x) u(x) v(x) \dd x \leqslant\mathcal E(u,v) + \int_{B_1^+} c^+(x) u(x) v(x) \dd x
\end{align*} and thereofore \begin{align*}
(-\Delta)^su +c^+u \geqslant 0 \qquad \text{in } B_1^+.
\end{align*} Hence, it suffices to prove Proposition~\ref{lem:FLIFu} with \(c \geqslant 0\). Let \(\rho>0\) be such that \(B_{2\rho} \subset B_1\) and  \(\varphi_\rho (x) = \varphi(x/\rho)\) where \(\varphi\) is as in Lemma \ref{SaBD4}. Provided that \(\varepsilon\) is sufficiently small, we have \begin{align*}
(-\Delta)^s(u-\varepsilon \varphi_\rho) +c(u-\varepsilon \varphi_\rho)\geqslant 0  \qquad \text{in }  B_{2\rho} \setminus B_{\rho/2}(\rho e_1)
\end{align*} and \(u-\varepsilon \varphi_\rho\geqslant 0\) in \((\R^n_+\setminus B_{2\rho})\cup B_{\rho/2}(\rho e_1)\). It follows from Proposition~\ref{aKht4} that \begin{align*}
u \geqslant \varepsilon \varphi_\rho \qquad \text{in } \R^n_+
\end{align*} where we used that \(c\geqslant0\). Since \(u(0)=\varphi_\rho(0)=0\), we conclude that \begin{align*}
\liminf_{h\to 0} \frac{u(he_1)}{h} \geqslant \varepsilon \partial_1 \varphi_\rho(0) >0,
\end{align*}
as desired.
\end{proof}

\section{Symmetry for the semilinear fractional parallel surface problem}\label{Svlif}

In this section, we will give the proof of Theorem \ref{CccFw}. For simplicity and the convenience of the reader, we will first state the particular case of Proposition~3.1 in \cite{MR3395749}, which we make use of several times in the proof of Theorem \ref{CccFw}. Note that \(c\) in Proposition \ref{QawmgWdG} corresponds to \(-c\) in \cite{MR3395749}---we made this change so that the notation of Proposition \ref{QawmgWdG} would agree with the notation in the proof of Theorem \ref{CccFw}.

\begin{prop}[Proposition~3.1 in \cite{MR3395749}] \label{QawmgWdG} Let \(H\) be a halfspace, \(\Omega \subset H\) be any open, bounded set, and \(c\in L^\infty (\Omega)\) be such that \(-c \leqslant c_\infty<\lambda_1(\Omega)\) in \(\Omega\) for some \(c_\infty \geqslant 0\),
where~\(\lambda_1(\Omega)\) is the first Dirichlet eigenvalue of \((-\Delta)^s\) in \(\Omega\). If \(u \in H^s(\R^n)\) satisfies \((-\Delta)^s u +cu \geqslant 0\) in \(\Omega\) and \(u\) is antisymmetric in \(\R^n\) then \(u \geqslant 0\) almost everywhere in \(\Omega\). 
\end{prop}

Now we will prove Theorem \ref{CccFw} in the case \(n\geqslant2\). We prove the case \(n=1\) later in the section.

\begin{proof}[Proof of Theorem \ref{CccFw} for \(n\geqslant2\)] 
Fix a direction \( e \in \Sph^{n-1}\). Without loss of generality, we may assume that~\(e=e_1\). Let \(T_\lambda = \{x\in \R^n \text{ s.t. } x_1 = \lambda\}\)---this will be our `moving plane' which we will vary by decreasing the value of \(\lambda\). Since \(\Omega\) is bounded, we may let \(M = \sup_{x\in \Omega}x_1\) which is the first value of \(\lambda\) for which \(T_\lambda\) intersects \(\overline{\Omega}\). Moreover, let \(H_\lambda = \{ x \in \R^n \text{ s.t. } x_1 >\lambda\}\), \(\Omega_\lambda = H_\lambda \cap \Omega\), and \(Q_\lambda:\R^n \to \R^n\) be given by \(x \mapsto x- 2x_1+2\lambda e_1\). Geometrically, \(Q_\lambda(x)\) is the reflection of \(x\) across the hyperplane~\(T_\lambda\). 

Since \(\partial \Omega \) is \(C^1\) there exists some \(\mu <M\) such that \(\Omega_\lambda' := Q_\lambda(\Omega_\lambda) \subset \Omega\) for all \(\lambda\in (\mu,M)\), see for example \cite{MR1751289}. Let \(m\) be the smallest such \(\mu\), that is, let \begin{align*}
m = \inf \{ \mu<M \text{ s.t. } \Omega_\lambda' \subset \Omega \text{ for all } \lambda \in  (\mu,M)\}.
\end{align*} Note that \(\Omega_m' \subset \Omega\). Indeed, if this were not the case then there would exist some \(\varepsilon>0\) such that \(\Omega_{m+\varepsilon}' \not\subset \Omega\), which would contradict the definition of~\(m\). 

As is standard in the method of moving planes, we will consider the function \begin{align*}
v_\lambda (x) &= u(x) - u(Q_\lambda(x)), \qquad x\in \R^n
\end{align*}  for each \(\lambda \in [m,M)\). It follows that \(v_\lambda\) is antisymmetric and satisfies \begin{align*}
\begin{PDE}
(-\Delta)^s v_\lambda +c_\lambda v_\lambda &= 0 &\text{in } \Omega_\lambda' ,\\
v_\lambda &\geqslant 0 &\text{in } H_\lambda'\setminus \Omega_\lambda' ,
\end{PDE}
\end{align*} where \begin{align*}
c_\lambda (x) &= \begin{cases}
\frac{f((u(x))-f(u(Q_\lambda(x)))}{u(x) - u(Q_\lambda(x))}, &\text{if } u(x) \neq u(Q_\lambda(x)), \\
0,& \text{if } u(x) = u(Q_\lambda(x)),
\end{cases}
\end{align*} and \(H_\lambda' := Q_\lambda (H_\lambda )\). We claim that \begin{align}
 v_m \equiv 0 \qquad \text{in } \R^n. \label{VxU8s}
\end{align} 

Before we show \eqref{VxU8s}, let us first prove the weaker statement \begin{align}
v_m \geqslant 0 \qquad \text{in } H_m' . \label{g3zkO}
\end{align}  Since \(f \in C^{0,1}_{\textrm{loc}} (\R)\), it follows that \(c_\lambda \in L^\infty (\Omega_\lambda')\) and that \begin{align*}
\|c_\lambda \|_{L^\infty (\Omega_\lambda')} \leqslant  [ f ]_{C^{0,1}([0,\| u \|_{L^\infty(\Omega)}])} .
\end{align*} Here, as usual, \begin{align*}
[ f ]_{C^{0,1}([0,a])} &= \sup_{\substack{x,y\in [0,a]\\x\neq y} } \frac{\vert f(x)-f(y)\vert }{\vert x - y \vert}.
\end{align*}  We cannot directly apply Proposition~\ref{QawmgWdG} as \([ f ]_{C^{0,1}([0,\| u \|_{L^\infty(\Omega)}])}\) might be large. However,
%
%
by Proposition~\ref{ayQzh}
we have that \begin{align*}
\lambda_1 (\Omega_\lambda' ) \geqslant C \vert \Omega_\lambda' \vert^{-\frac{2s}n} \to \infty
\quad \text{ as } \, \lambda \to M^- ,
\end{align*} 
%
%
and since \(\|c_\lambda \|_{L^\infty (\Omega_\lambda')}\) is uniformly bounded with respect to \(\lambda\), Proposition~\ref{QawmgWdG} implies that there exists some \(\mu \in [m,M)\) such that \(v_\lambda \geqslant 0\) in \(\Omega_\lambda'\) for all \(\lambda \in [\mu,M)\). In fact, since \(u\) is not identically zero, after possibly increasing the value of \(\mu\) (still with \(\mu<M\)) we claim that \(v_\lambda > 0\) in \(\Omega_\lambda'\) for all \(\lambda \in [\mu,M)\). Indeed, there exists \(x_0\in \Omega\) such that \(u(x_0)>0\) so, provided \(\mu\) is close to \(M\), \(Q_\lambda (x_0) \not\in \Omega\). Then, for all \(\lambda \in [\mu,M)\), \begin{align*}
v_\lambda (x_0) = u(x_0) >0
\end{align*} so it follows from the strong maximum principle Proposition~\ref{fyoaW} that \(v_\lambda>0\) in \(\Omega_\lambda'\).

This allows us to define \begin{align*}
\tilde{m} =\inf \{ \mu\in [m,M) \text{ s.t. } v_\lambda > 0 \text{ in } \Omega_\lambda' \text{ for all } \lambda \in [\mu,M)\}.
\end{align*} We claim that \(\tilde{m} = m\). For the sake of contradiction, suppose that we have \(\tilde{m} >m\). By continuity in \(\lambda\), \(v_{\tilde{m}} \geqslant 0\) in \(H_{\tilde{m}}'\) and then, as above, Proposition~\ref{fyoaW} implies that \(v_{\tilde{m}} >0\) in \(\Omega_{\tilde{m}}'\). Due to the definition of \(\tilde{m}\), for all \(0<\varepsilon < \tilde{m}-m\), the set \(\{ v_{\tilde m-\varepsilon} \leqslant 0\} \cap \Omega_{\tilde m-\varepsilon}' \) is non-empty. Let \(\Pi_\varepsilon \subset\Omega_{\tilde m-\varepsilon}'\) be an open set such that \( \{ v_{\tilde m-\varepsilon} \leqslant 0\} \cap \Omega_{\tilde m-\varepsilon}'  \subset \Pi_\varepsilon\). By making \(\varepsilon\) smaller we may choose \(\Pi_\varepsilon\) such that \(\vert \Pi_\varepsilon\vert \) is arbitrarily close to zero. Hence, applying Proposition~\ref{QawmgWdG} then Proposition~\ref{fyoaW} gives that \(v_{\tilde m -\varepsilon} >0\) in \(\Pi_\varepsilon\) which is a contradiction. This proves \eqref{g3zkO}. 

Since \(v_m \geqslant 0\) in \(H_m'\), Proposition~\ref{fyoaW} implies that either \eqref{VxU8s} holds or \(v_m>0\) in \(\Omega_m'\). For the sake of contradiction, let us suppose that \begin{align}
v_m>0 \qquad \text{in } \Omega_m'. \label{JQNtW}
\end{align} By definition, \(m\) is the minimum value of \(\lambda\) for which \(\Omega_\lambda' \subset \Omega\). There are only two possible cases that can occur at this point. 

\begin{quotation}
Case 1: There exists \(p \in ( \Omega_m' \cap \partial \Omega) \setminus T_m\).\\
Case 2: There exists \(p \in T_m \cap \partial \Omega\) such that \(e_1\) is tangent to \(\partial \Omega\) at \(p\).
\end{quotation} 

If we are in Case 1 then there is a corresponding point \(q\in ( G_m' \cap \partial G) \setminus T_m \subset \Omega_m'\) since \(\partial G\) is parallel to \(\partial \Omega\). But \(u\) is constant on \(\partial G\), so we have \begin{align*}
v_m(q) = u(q) - u(Q_m (q)) = 0
\end{align*} which contradicts \eqref{JQNtW}. 

If we are in Case 2 then there exists \(q \in T_m \cap \partial G \) such that \(e_1\) is tangent to \(\partial G\) at \(q\). Since \(u\) is a constant on \(\partial G\), the gradient of \(u\) is perpendicular to \(\partial G\) and so \(\partial_1 u(q) = 0. \) Moreover, by the chain rule, \begin{align*}
\partial_1 (u\circ Q_m)(q) = - \partial_1u(q) = 0. 
\end{align*} Hence, \begin{align*}
\partial_1 v_m (q) = 0. 
\end{align*} However, this contradicts Proposition~\ref{lem:FLIFu}. Thus, in both Case 1 and Case 2 we have shown that the assumption \eqref{JQNtW} leads to a contradiction, so we conclude that \eqref{VxU8s} is true. 

This concludes the main part of the proof. The final two steps are to show that \(u\) is radially symmetric and that \(\supp u = \overline{\Omega}\). Due to the previous arguments we know that for all \(e\in \Sph^{n-1}\) and \(\lambda \in [m(e), M(e))\), \begin{align*}
u(x) - u(Q_{\lambda,e}(x))\geqslant 0 \qquad \text{in } H_{\lambda,e}'
\end{align*} where \(Q_{\lambda,e} =Q_{\lambda}\), \(H'_{\lambda,e} =H_{\lambda}' \) as before but we've included the subscript \(e\) to emphasise that the result is true for each \(e\). It follows that for every halfspace we either have \(u(x) - u(Q_{\partial H}(x)) \geqslant 0\) for all \(x\in H\) or \(u(x) - u(Q_{\partial H}(x)) \leqslant 0\) for all \(x\in H\). By \cite[Proposition 2.3]{MR2722502}, we can conclude that there exists some \(z\in \R^n\) such that \(x \mapsto u(x-z)\) is radially symmetric. Moreover, since \(u\geqslant0\) in \(\Omega\) and not identically zero, \cite[Proposition 2.3]{MR2722502} also implies that \(x \mapsto u(x-z)\) is non-increasing in the radial direction. 

It follows that \(\supp u\) is a closed ball. (Note that we are using the convention \\\(\supp u = \overline{\{x \in \R^n \text{ s.t. } u(x)\neq 0 \}}\)). We claim that \(\supp u = \overline{\Omega}\). For the sake of contradiction, suppose that this is not the case. Then there exists a direction \(e\in \Sph^{n-1}\) and \(\lambda \in (m(e),M(e))\) such that \(\Omega'_\lambda \cap \supp u = \varnothing \). It follows that \(v_\lambda \equiv 0\) in \(\Omega_\lambda'\). This is a contradiction since we previously showed that \(v_\lambda>0\) in \( \Omega_\lambda'\) for all \(\lambda \in (m(e),M(e))\).
\end{proof}

For $n=1$, Theorem \ref{CccFw} reads as follows.

\begin{thm}  Suppose that \(G\) is a bounded open set in \(\R\),  \(\Omega=G+B_R\), \(f:\R \to \R\) is locally Lipschitz, and \(c_0\in \R\). Furthermore, assume that there exists a non-negative function \(u \in C^s(\R)\) that is not identically zero and satisfies\begin{align} 
\begin{PDE}
(-\Delta)^s u &= f(u) &\text{in } \Omega , \\ u&=0 &\text{in } \R \setminus \Omega , \\ u &= c_0 & \text{on } \partial G .
\end{PDE} 
\end{align} Then, up to a translation, \(u\) is even, \(u>0\) in \(\Omega\), and \(G=(a,b)\) for some \(a<b\). 
\end{thm}

\begin{proof}
In one dimension the moving plane \(T_\lambda\) is just a point \(\lambda\) and there are only two directions \(T_\lambda\) can move--from right to left or from left to right. We will begin by considering the case \(T_\lambda\) is moving from right to left. Let \((a-R, b +R)\), \(a<b\), be the connected component of \(\Omega\) such that \(\sup \Omega=b+R\). Using the same notation as in the proof of Theorem \ref{CccFw}, it is clear that \(M =b+R\) and \(m = \frac{a+b} 2 \). For all \(\lambda \in ( \frac{a+b} 2, b+R)\), let  \begin{align*}
v_\lambda (x) &= u(x) - u(-x+2 \lambda ), \qquad x\in \R. 
\end{align*} Arguing as in the proof of Theorem \ref{CccFw}, we obtain \begin{align}
v_\lambda &> 0 \qquad \text{in } (-\infty , \lambda) \label{8e2sD}
\end{align} for all \(\lambda \in  ( \frac{a+b} 2, b+R)\) and \begin{align*}
v_{\frac{a+b}2} &\geqslant 0 \qquad \text{in } (-\infty , (a+b)/2).
\end{align*} Then the overdetermined condition \(u(a)=u(b)=c_0\) implies that \begin{align*}
v_{\frac{a+b}2} (b)&= u(b) - u(a) =0.
\end{align*} Hence, Proposition~\ref{fyoaW} gives that \begin{align}
v_{\frac{a+b}2} &\equiv 0 \qquad \text{in } \R, \label{os0w4}
\end{align}that is, \(u\) is an even function about the point \(\frac{a+b}2\). Moreover, \eqref{os0w4} implies that \(u\equiv 0\) in \(\R\setminus (a-R,b+R)\). 

Now suppose that we move \(T_\lambda\) from left to right. If \(\Omega\) is made up of at least two connected components then repeating the argument above we must also have that \(u \equiv 0\) in \((a-R,b+R)\). However, \(u\) is not identically zero, so \(\Omega\) can only have one connected component. Finally, \eqref{8e2sD} implies that \(u\) is strictly monotone in \((\frac{a+b} 2 , b+R)\) which further implies the positivity of \(u\).
\end{proof}

\appendix

\section{Technical lemmas} \label{ltalz}

In this appendix, we list several lemmas that are used throughout the paper.

\begin{lem} \label{IyXUO}
Let \(I \subset \R\) and \(\Omega \subset \R^n\) be open and bounded sets in \(\R\) and \( \R^n\) respectively. Suppose that \(\phi_t(x): \overline{\Omega}\times \overline{I} \to \R\) is continuous in \(  \overline{\Omega} \times \overline{I} \) and define \begin{align*}
m(t) = \min_{x \in \overline{\Omega}} \phi_t(x) . 
\end{align*} Then \(m \in C(\overline{I})\). 
\end{lem}

\begin{proof}
Since \(  \overline{\Omega} \times  \overline{I}\) is compact, \(\phi\) is uniformly continuous. Fix \(\varepsilon >0\) and let \((\bar{x},\bar{t}) \in \overline{\Omega} \times  \overline{I}\) be arbitrary. There exists some \(\delta>0\) (indepedent of \(\bar{t}\) and \(\bar{x}\)) such that if \((x,t) \in \overline{\Omega} \times  \overline{I}\) and \( \abs{(x,t)-(\bar{x},\bar{t})} < \delta \) then \begin{align*}
\abs{\phi_t(x)-\phi_{\bar{t}}(\bar{x})} < \varepsilon .
\end{align*} In particular, we may take \(x=\bar{x}\) to conclude that if \(t \in \overline{I}\) and \(\abs{t-\bar{t}} < \delta\) then \begin{align*}
\abs{\phi_t(\bar{x})-\phi_{\bar{t}}(\bar{x})} < \varepsilon 
\end{align*} for any \(\bar{x} \in \overline{\Omega}\). Consequently, \begin{align*}
\phi_t(\bar{x}) > \phi_{\bar{t}}(\bar{x}) - \varepsilon \geqslant m(\bar{t}) -\varepsilon \qquad \text{for all } \bar{x} \in \overline{\Omega}
\end{align*} Minimising over \(\bar{x}\) we obtain \begin{align*}
m(t) \geqslant  m(\bar{t}) -\varepsilon .
\end{align*} Similarly, we also have that \begin{align*}
m(\bar{t}) \geqslant  m(t) -\varepsilon . 
\end{align*} Thus, we have shown that if \(t \in \overline{I}\) and \(\abs{t-\bar{t} } < \delta\) then \begin{equation*}
\abs{m(t) -m(\bar{t} )} \leqslant \varepsilon . \qedhere
\end{equation*}
\end{proof}

\begin{lem} \label{oujju}
Let \(R>0\) and \(\zeta_R: \R \to \R \) be the solution to \begin{align}
\begin{PDE}
(-\Delta)^s \zeta_R &=0 &\text{in } (-R,R) ,\\
\zeta_R &= g_R &\text{in } \R \setminus (-R,R),
\end{PDE} \label{acLUt}
\end{align} where \begin{align*}
g_R(x) &= \begin{cases}
R, &\text{if }x >R ,\\
-R, &\text{if } x<-R . 
\end{cases}
\end{align*} Then \begin{align}
x \mapsto \frac{\zeta_R(x)}{x} \qquad \text{is  defined at }x=0 \label{JQM6p}
\end{align} and there exists a constant \(c_0=c_0(s)>0\) such that, as \(R\to \infty\), \begin{align*}
\frac{\zeta_R(x)}{x} \to c_0  \qquad \text{in } C_{\textrm{loc}} (\R).
\end{align*} 
\end{lem}

\begin{proof} By the scale-invariance property of the fractional Laplacian, we may write \begin{align}
\zeta_R(x) &= R \zeta_1(x/R) \label{DkRUV}
\end{align} where \(\zeta_1\) is the solution to \eqref{acLUt} with \(R=1\). 
The function \(\zeta_1\) is given explicitly via the Poisson kernel, (for more details see \cite[Section 15]{MR3916700}) \begin{align*}
\zeta_1 (x) &= a_s\big ( 1 -x^2 \big)^s \int_{\R \setminus (-1,1)} \frac{g_1(t)}{(t^2-1)^s\vert x-t \vert} \dd t, \qquad x \in (-1,1)  .
\end{align*} where \(a_s\) is a positive normalisation constant. Using the definition of \(g_1\) and a change of variables, we obtain  \begin{align}
\zeta_1( x) = 2a_s x\big (1-x^2 \big )^s \int_1^\infty \frac{\dd t }{(t^2-x^2)(t^2-1)^s} . \label{XLPcN}
\end{align} Provided \(\vert x\vert <1/2\), \begin{align*}
\frac{1}{(t^2-x^2)(t^2-1)^s} \leqslant \frac 1{(t^2-(1/2)^2)(t^2-1)^s} \in L^1((1,\infty))
\end{align*} so by \eqref{XLPcN} and the dominated convergence theorem \begin{align*}
\lim_{x\to 0} \frac{\zeta_R(x)}{x} &= R\lim_{x\to 0} \frac{\zeta_1(x/R)}{x}= 2a_s  \int_1^\infty \frac{\dd t }{t^2(t^2-1)^s} =:c_0.
\end{align*} This proves \eqref{JQM6p}. 

Moreover, \begin{align}
\bigg \vert \frac{\zeta_R(x)}{x} -  c_0\bigg \vert &= 2a_s\bigg \vert    \int_1^\infty \bigg ( \frac {\big (1-(x/R)^2 \big )^s} {t^2-(x/R)^2} - \frac 1 {t^2} \bigg ) \frac{\dd t }{(t^2-1)^s} \bigg \vert \nonumber \\
&\leqslant 2a_s   \int_1^\infty \bigg \vert  \frac {t^2\big ( (1-(x/R)^2  )^s-1\big ) +(x/R)^2 \big )} {t^2\big (t^2-(x/R)^2 \big )}  \bigg \vert \frac{\dd t }{(t^2-1)^s} \nonumber \\
&\leqslant 2a_s \int_1^\infty \bigg \vert \frac{\big (1-(x/R)^2\big )^s -1}{t^2-(x/R)^2} \bigg \vert \frac{\dd t }{(t^2-1)^s} \nonumber\\
&\qquad + \frac{2a_s\vert x \vert^2}{R^2}\int_1^\infty \frac{\dd t}{t^2 \vert t^2-(x/R)^2 \vert (t^2-1)^s} . \label{ZRXKr}
\end{align} Suppose that \(x \in  \Omega' \subset \subset \R\). Then for \(R>0\) sufficiently large, \begin{align}
\frac 1 {\vert t^2-(x/R)^2  \vert } \leqslant C \qquad \text{for all } t>1 \label{RZqcI}
\end{align} and, by Bernoulli's inequality, \begin{align}
\vert \big (1-(x/R)^2\big )^s -1 \vert &\leqslant \frac C {R^2}. \label{zHcKY}
\end{align} Combining \eqref{ZRXKr}, \eqref{RZqcI}, and \eqref{zHcKY}, we conclude that \begin{align*}
\bigg \vert \frac{\zeta_R(x)}{x} - c_0\bigg \vert &\leqslant \frac C {R^2}\bigg ( \int_1^\infty  \frac{\dd t }{(t^2-1)^s}
+ \int_1^\infty \frac{\dd t}{t^2(t^2-1)^s} \bigg ) \to 0
\end{align*} as \(R \to \infty\).
\end{proof}

\begin{prop} \label{ayQzh}
Suppose that \(\Omega\) is a bounded open subset of \(\R^n\) and that \(\lambda(\Omega)\) is a Dirichlet eigenvalue of \((-\Delta)^s\) in \(\Omega\). Then \begin{align*}
\lambda(\Omega)\geqslant \frac{n}{2s} \vert B_1 \vert^{1+2s/n} c_{n,s}\vert \Omega \vert^{- \frac{2s}n } .
\end{align*} where \(c_{n,s}\) is the defined in \eqref{tK7sb}.
\end{prop}

The result in Proposition \ref{ayQzh} is not new (see \cite{MR3824213,MR3063552}); however, to the authors' knowledge, the proof is original and simple.

\begin{proof}[Proof of Proposition \ref{ayQzh}]
Let \(u \in L^2(\Omega)\) satisfy \begin{align*}
\begin{PDE}
(-\Delta)^s u &= \lambda(\Omega) u &\text{in } \Omega, \\
u&= 0 &\text{in }\R^n \setminus \Omega, \\
\| u \|_{L^2(\Omega)}&=1. 
\end{PDE}
\end{align*} The existence of such a function can be proved via semi-group theory, see for example \cite[Chapter 4]{MR2569321}.

Using the integration by parts formula for the fractional Laplacian, see \cite[Lemma 3.3]{Dipierro2017Neum}, we have
\begin{align*}
\lambda(\Omega) = \frac{c_{n,s}}{2} \int_{\R^{2n} \setminus (\Omega^c)^2} \frac{\vert u(x)-u(y) \vert^2}{\vert x- y \vert^{n+2s}} \dd y \dd x \geqslant c_{n,s} \int_\Omega \int_{ \R^n \setminus \Omega } \frac{\vert u(x)\vert^2}{\vert x- y \vert^{n+2s}} \dd y \dd x. 
\end{align*}For \(x\in \Omega\), if \(B_r(x)\) is the ball such that \(\vert B_r(x)\vert = \vert \Omega\vert \) then by~\cite[Lemma 6.1]{MR2944369}, \begin{align*}
\int_{ \R^n \setminus \Omega } \frac{ \dd y }{\vert x- y \vert^{n+2s}} \geqslant \int_{ \R^n \setminus B_r(x)} \frac{ \dd y }{\vert x- y \vert^{n+2s}} =\frac{n}{2s} \vert B_1 \vert^{1+2s/n} \vert \Omega\vert^{-\frac{2s}{n}}.
\end{align*} Hence, \begin{equation*}
\lambda(\Omega) \geqslant \frac{n}{2s} \vert B_1 \vert^{1+2s/n}c_{n,s}\vert \Omega\vert^{-\frac{2s}{n}} \| u \|_{L^2(\Omega)}^2 =  \frac{n}{2s} \vert B_1 \vert^{1+2s/n}c_{n,s} \vert \Omega\vert^{-\frac{2s}{n}}.\qedhere
\end{equation*}
\end{proof}

\section*{Acknowledgments}

All the authors are members of AustMS. SD is
supported by the Australian Research Council DECRA DE180100957 “PDEs, free boundaries and
applications”.
GP is supported by the Australian Research Council (ARC) Discovery Early Career Researcher Award (DECRA) DE230100954 “Partial Differential Equations: geometric aspects and applications”, and is member of INdAM/GNAMPA.
GP and EV are supported by the Australian Laureate Fellowship FL190100081 “Minimal surfaces, free boundaries and partial differential equations”.
JT is supported by an Australian Government Research Training Program Scholarship. 

\printbibliography
 
 \vfill
\end{document}